\documentclass[12pt]{amsart}

\usepackage[english]{babel}
\usepackage[utf8]{inputenc}

\usepackage{amsfonts,amssymb,amscd,amstext,mathrsfs}
\usepackage{verbatim}
\usepackage{enumerate}
\input xy
\xyoption{all}

\newtheorem{theorem}{Theorem}[section]
\newtheorem{proposition}[theorem]{Proposition}

\newtheorem{lemma}[theorem]{Lemma}

\theoremstyle{definition}

\newtheorem{remark}[theorem]{Remark}

\newtheorem{problem}[theorem]{Problem}

\numberwithin{equation}{section}
\numberwithin{figure}{section}

\newcommand\Ecal{\mathcal{E}}

\newcommand\C{\mathbb{C}}

\newcommand\CP{\mathbb{CP}}

\newcommand\N{\mathbb{N}}

\newcommand\U{\mathbb{U}}

\newcommand\igot{\mathfrak{i}}

\renewcommand\igot{\mathfrak{i}}

\renewcommand\imath{\igot}

\newcommand\wt{\widetilde}

\newcommand\di{\partial}

\newcommand\Aut{\mathrm{Aut}}

\newcommand\Hom{\mathrm{Hom}}

\def\Ell1{\mathrm{Ell_1}}
\def\DEll1{\mathrm{DEll_1}}

\numberwithin{equation}{section}

\begin{document}
\title{Every projective Oka manifold is elliptic}
\author{Franc Forstneri\v{c} and Finnur L\'arusson}

\address{Franc Forstneri\v c, Faculty of Mathematics and Physics, University of Ljubljana, Jadranska 19, 1000 Ljubljana, Slovenia}

\address{Franc Forstneri\v c, Institute of Mathematics, Physics, and Mechanics, Jadranska 19, 1000 Ljubljana, Slovenia}
\email{franc.forstneric@fmf.uni-lj.si}

\address{Finnur L\'arusson, Discipline of Mathematical Sciences, University of Adelaide, Adelaide SA 5005, Australia}
\email{finnur.larusson@adelaide.edu.au}


\date{29 March 2025.  Most recent edits 4 May 2026}


\begin{abstract}
We show that every projective Oka manifold
is elliptic in the sense of Gromov. This gives an affirmative answer
to a long-standing open question.
\end{abstract}

\maketitle



%
%
%
%
\section{Introduction}\label{sec:intro} 

A complex manifold $Y$ is said to be an Oka manifold
if it satisfies all forms of the homotopy principle 
(also called the Oka principle in this context) for holomorphic maps $X\to Y$
from any Stein manifold $X$. One of the simplest characterisations
of this class of manifolds is the convex approximation 
property introduced in \cite{Forstneric2006AM}; see
also \cite[Sect.\ 5.4]{Forstneric2017E}.
In Gromov's terminology \cite[3.1, p.\ 878]{Gromov1989}, 
Oka manifolds are called $\mathrm{Ell}_\infty$ manifolds.

A complex manifold $Y$ is said to be elliptic if it admits
a dominating holomorphic spray $s:E\to Y$ defined on the
total space of a complex vector bundle $\pi:E\to Y$
\cite[0.5, p.\ 855]{Gromov1989}. 
This means that $s$ restricts to the identity map on the zero 
section $E_0\cong Y$ of $E$, and for every $y\in Y$, the differential
$ds_{0_y}$ at the origin $0_y\in E_y=\pi^{-1}(y)$ maps 
the fibre $E_y$ onto $T_yY$.
An ostensibly weaker condition, called subellipticity, was 
introduced by the first-named author in \cite[Definition 2]{Forstneric2002MZ}.
It asks for the existence of finitely many holomorphic
sprays $(E_j,\pi_j,s_j)$ on $Y$ $(j=1,\ldots,m)$ that are dominating, 
in the sense that
\begin{equation}
\label{eqn:domination}
    (ds_1)_{0_y}(E_{1,y}) + (ds_2)_{0_y}(E_{2,y})+\cdots\\
                     + (ds_m)_{0_y}(E_{m,y})= T_y Y
                     \quad \text{for all $y\in Y$}.
\end{equation}
One of the main results of Oka theory is that every  
elliptic manifold is an Oka manifold 
(see Gromov \cite[0.6, p.\ 855]{Gromov1989}
and \cite{ForstnericPrezelj2000}), and every subelliptic 
manifold is an Oka manifold (see \cite[Theorem 1.1]{Forstneric2002MZ}). 
For a comprehensive survey, see \cite[Chap.\ 5]{Forstneric2017E}.
Examples of elliptic and subelliptic manifolds can be found 
in \cite[Sect.\ 6.4]{Forstneric2017E} and in the surveys \cite{Forstneric2023Indag,Forstneric2013KAWA,ForstnericLarusson2011}.
Every complex homogeneous manifold
is elliptic but not conversely.  Several recent results are mentioned below.

In this paper we prove the following main result.

%
%
\begin{theorem}\label{th:elliptic}
Every projective Oka manifold is elliptic.
\end{theorem}

It follows that, for projective manifolds, the Oka property, ellipticity, and subellipticity are equivalent.
The spray bundle on a projective Oka manifold $Y$ that emerges in the proof of the theorem is easily described: it is the direct sum of some number of copies of the dual of a sufficiently ample line bundle on $Y$.  In a suitable embedding of $Y$ into complex projective space, the spray bundle is therefore the direct sum of copies of the universal line bundle.

Theorem \ref{th:elliptic} solves a long-standing open problem, 
originating in Gromov's 1989 paper \cite[3.2.A" Question]{Gromov1989},  
whether every Oka manifold is elliptic;
see also \cite[Problem 6.4.21]{Forstneric2017E},
where the analogous question was posed for subellipticity.
The first counterexamples to both questions for noncompact manifolds
were found only very recently. 
In 2024, Kusakabe showed that the complement
$\C^n\setminus K$ of any compact polynomially convex set 
$K\subset\C^n$ for $n\geq 2$ is an Oka manifold
\cite[Theorem 1.6]{Kusakabe2024AM}. 
A few years earlier it was shown by 
Andrist, Shcherbina, and Wold \cite{AndristShcherbinaWold2016}
that if $K$ is a compact set with nonempty interior
in $\C^n$ for $n\ge 3$, then $\C^n\setminus K$ 
fails to be subelliptic.
Taking $K$ to also be polynomially convex, it follows that 
$\C^n\setminus K$ is Oka but not subelliptic.
These examples are not Stein.  As observed by Gromov, every Stein 
Oka manifold is elliptic \cite[3.2.A, p.\ 879]{Gromov1989}. 

In light of Theorem \ref{th:elliptic}, the main remaining
question on this topic is the following.

\begin{problem}
Is there a compact non-projective Oka manifold that fails
to be elliptic or subelliptic?
\end{problem}

%
%
A much-studied property of algebraic manifolds
is the algebraic version of ellipticity.
A complex algebraic manifold $Y$ is said to be algebraically
elliptic if it admits an algebraic dominating spray $s:E\to Y$
defined on the total space of an algebraic vector bundle
$\pi:E\to Y$; see \cite[Definition 5.6.13 (e)]{Forstneric2017E}.
Similarly, $Y$ is algebraically subelliptic if it admits 
finitely many algebraic sprays $(E_j,\pi_j,s_j)$ satisfying 
\eqref{eqn:domination}. It was recently shown by Kaliman and Zaidenberg
\cite{KalimanZaidenberg2024FM} that every algebraically subelliptic 
manifold is algebraically elliptic; the converse is a tautology. 
Algebraic ellipticity is a Zariski-local condition as shown by 
Gromov \cite[3.5.B, 3.5.C]{Gromov1989}; see also 
\cite[Proposition 6.4.2]{Forstneric2017E}.
No such results are known in the holomorphic category.
A recent result of Banecki \cite{Banecki2024X} is that 
every rational projective manifold is algebraically elliptic.
This generalises the result of Arzhantsev, Kaliman, and Zaidenberg 
\cite[Theorem 1.3]{ArzhantsevKalimanZaidenberg2024}
that every uniformly rational projective manifold is algebraically
elliptic. See also the recent examples of algebraically elliptic
projective manifolds in 
\cite{KalimanZaidenberg2024MRL,KalimanZaidenberg2025}
and the survey \cite{Zaidenberg2024ellipticitysurvey}.
Every algebraically elliptic manifold $Y$ satisfies the algebraic 
homotopy approximation theorem for maps $X\to Y$ from 
affine algebraic manifolds $X$,
showing in particular that every holomorphic map that is homotopic
to an algebraic map is a limit of algebraic maps in the compact-open
topology; see \cite[Theorem 3.1]{Forstneric2006AJM}, 
\cite[Theorem 6.15.1]{Forstneric2017E}, and the recent 
generalisations in \cite[Sect.\ 2]{AlarconForstnericLarusson2024}.
As shown by L\'arusson and Truong \cite{LarussonTruong2019},
this is the closest analogue of the Oka principle in the algebraic
category. However, there are examples of projective Oka manifolds 
that fail to be algebraically elliptic, for example, abelian varieties. 
Hence, the algebraic counterpart
to Theorem \ref{th:elliptic} is not true, and the GAGA principle
of Serre \cite{Serre1955AIF} fails for ellipticity of projective manifolds.

Besides its intrinsic importance in Oka theory, Theorem \ref{th:elliptic}
is interesting for the following reason. For a long time, essentially 
the only known examples of Oka manifolds were the elliptic and subelliptic 
manifolds. Thanks to recent developments, we now have at our disposal 
several other methods to discover Oka manifolds. 
In particular, Kusakabe's localisation theorem
\cite[Theorem 1.4]{Kusakabe2021IUMJ} says that a complex manifold
covered by Zariski-open (in the holomorphic sense)
Oka domains is an Oka manifold.
No such localisation result is available for holomorphic ellipticity or subellipticity, 
so it is interesting that we nevertheless get ellipticity from the Oka property 
in the class of projective manifolds.

%
%
\section{Proof of Theorem \ref{th:elliptic}}\label{sec:proof}

Let $Y$ be a projective manifold, embedded in $n$-dimensional complex projective space $\CP^n$.
Let $z=[z_0:z_1:\cdots:z_n]$ be homogeneous coordinates
on $\CP^n$. Set $\Lambda_\alpha=\{z_\alpha=0\}$ for $\alpha=0,1,\ldots,n$, 
and let $U_\alpha=\CP^n\setminus \Lambda_\alpha\cong\C^n$ be the affine
chart with coordinates $(z_0/z_\alpha,\ldots,z_n/z_\alpha)$, 
where the term $z_\alpha/z_\alpha=1$ is omitted. 
Denote the affine coordinates on $U_0$ by $x=(x_1,\ldots,x_n)$,
with $x_i=z_i/z_0$. Let $\di_{x_i} = \di/\di x_i$ denote the coordinate
vector fields on $U_0\cong\C^n$.
Since $Y_0=Y\cap U_0$ is an algebraic submanifold
of $U_0\cong\C^n$, Serre's Theorem A
gives finitely many polynomial vector fields 
\begin{equation}\label{eq:Wj}
	W_j(x) = \sum_{i=1}^n V_{i,j}(x) \di_{x_i},\quad j=1,\ldots,m,
\end{equation} 
on $U_0\cong \C^n$ whose restrictions to $Y_0$ are tangent to $Y_0$ and  
span the tangent space $T_yY$ at every point $y\in Y_0$.
To this collection we associate the polynomial vector field $V$ on the 
total space $U_0\times \C^m\cong \C^{n+m}$
of the trivial vector bundle $\pi:U_0\times \C^m \to U_0$,
defined by
\begin{equation}\label{eq:V}
	V(x,t) = \sum_{i=1}^n \sum_{j=1}^m t_j \, V_{i,j}(x) \di_{x_i},
\end{equation}
where $x\in U_0$ and $t=(t_1,\ldots,t_m) \in \C^m$ is the fibre variable. 
Note that $V$ is horizontal in the sense that its $t$-component 
equals zero. Furthermore, $V$ vanishes on the zero section 
$U_0\times \{0\}^m=\{t=0\}$, 
and for every $(x,t) \in Y_0\times \C^m$ we have that  
\begin{equation}\label{eq:tangent}
	d\pi_{(x,t)} V(x,t) = 
	\sum_{i=1}^n \sum_{j=1}^m t_j\, V_{i,j}(x) \di_{x_i}
	= \sum_{j=1}^m t_j\, W_j(x) \in T_{x}Y.
\end{equation}
The formula without the last inclusion holds for all $(x,t)\in U_0\times \C^m$.

Recall that $\Lambda_0=\CP^n\setminus U_0=\{z_0=0\}$. 
Denote by $\U\to \CP^n$ the universal line bundle. 

\begin{lemma}\label{lem:V}
Let $k_0$ denote the maximum of the degrees of the polynomials $V_{i,j}(x)$.
For every $k\ge k_0$, the vector field $V$
\eqref{eq:V} extends to an algebraic vector field on the total space
of the vector bundle $E=(\CP^n \times \C^m)\otimes \U^k= m\U^k$ on 
$\CP^n$ (the direct sum of $m$ copies of the $k$-th tensor power of 
$\U$), which vanishes on the zero section
$E_0$ of $E$ and on $E|_{\Lambda_0}$. 
\end{lemma}

\begin{proof}
For every $\alpha=0,1,\ldots,n$ we have a vector bundle trivialisation 
$\theta_\alpha:E|U_\alpha\stackrel{\cong}{\to} U_\alpha\times \C^m$ 
with transition maps 
$\theta_{\alpha,\beta}=\theta_\alpha\circ \theta_\beta^{-1}$ on 
$(U_\alpha\cap U_\beta) \times \C^m$ given by 
\begin{equation}\label{eq:transition}
	\theta_{\alpha,\beta}([z],t) = \bigl([z],(z_\alpha/z_\beta)^k t\bigr),\quad 
	t\in\C^m, \ 0\le \alpha,\beta\le n.
\end{equation}
In particular, $\theta_{\alpha,0}([z],t)=\bigl([z],(z_\alpha/z_0)^k t\bigr)$.
We analyse the behaviour of the vector field $V$ \eqref{eq:V} 
near the hyperplane
$\Lambda_0\setminus \Lambda_\alpha$ for $\alpha=1,\ldots,n$. 
It suffices to consider the case $\alpha=1$ since the same argument 
will apply to every $\alpha$. 
For simplicity we first make the calculation in the case 
when $m=1$, so $E=\U^k$. The calculation is made in two steps. 
In the first step, we express $V$ in the fibre coordinate $t'$ on the 
line bundle chart $E|{U_1}\cong \C^n\times \C$ over the domain
$U_0\cap U_1 =\{x=(x_1,\ldots,x_n)\in \C^n: x_1\ne 0\}$.
The transition map $\theta_{1,0}$ is given by
$\theta_{1,0}(x,t)=\bigl(x,x_1^{k} t\bigr)$, so $t'=x_1^{k} t$. 
Its differential has the block form
\begin{equation}\label{eq:Dtheta}
	D\theta_{1,0}(x,t)=
	\begin{pmatrix}
	I_{n} & 0 \\
	B      & x_1^{k} 
	\end{pmatrix}
\end{equation}
where $I_n$ is the $n\times n$ identity matrix and
$B=(kx_1^{k-1} t,0,\ldots,0)$. It follows that the vector field
$V'=D\theta_{1,0} \, \cdotp V$ equals
\begin{align*} 
	V' &= t\sum_{i=1}^n V_i(x)\di_{x_i} + k t^2 x_1^{k-1}V_1(x) \di_{t'} \\
	   & = t' x_1^{-k}\sum_{i=1}^n V_i(x)\di_{x_i} + 
		k (t')^2 x_1^{-k-1} V_1(x) \di_{t'},
\end{align*} 
where we used that $t=x_1^{-k}t'$.
In the second step, we express the vector field $V'$ in the standard affine
coordinates $x'=(x'_1,x'_2,\ldots,x'_n)$ on $U_1$. Note that
\begin{align*}
	x'_1 &= \frac{z_0}{z_1} = \frac{1}{x_1},\\
	x'_i  &= \frac{z_i}{z_1} = \frac{x_i}{x_1},\quad i=2,\ldots,n. 	
\end{align*}
Write $x'=\psi(x)$ and $(x',t')=\tilde \psi(x,t')=(\psi(x),t')$. 
We have that 
\[
	D\psi(x)=
	\begin{pmatrix}
	-\frac{1}{x_1^2} & 0 & 0 & \cdots & 0 \\
	-\frac{x_2}{x_1^2} & \frac{1}{x_1} & 0 & \cdots & 0 \\
	\cdots & \cdots & \cdots & \cdots & \cdots  \\
	-\frac{x_n}{x_1^2} & 0 & 0 & \cdots & \frac{1}{x_1}
	\end{pmatrix}.
\]
Hence, the vector field $\wt V=D\tilde\psi\,\cdotp V'$ equals
\begin{align*}
	\wt V(x,t') &= -\frac{1}{x_1^{k+2}} V_1(x) t' \di_{x'_1} 
 	+\sum_{i=2}^n \left[ - \frac{x_i}{x_1^{k+2}} V_1(x) 
	+ \frac{1}{x_1^{k+1}} V_i(x)\right] t' \di_{x'_i} \\
	&\quad + \frac{k}{x_1^{k+1}} V_1(x) (t')^2 \di_{t'}. 
\end{align*}	
Note that $x=\psi^{-1}(x')=(1/x'_1,x'_2/x'_1,\cdots,x'_n/x'_1)$.
Inserting this in the above expression gives  
\begin{align*}\label{eq:Vtilde}
	\wt V(x',t') &= -(x'_1)^{k+2} V_1(\psi^{-1}(x')) \, t' \di_{x'_1} \\ 
	&\quad	+\sum_{i=2}^n (x'_1)^{k+1}  \left[ - x'_i V_1(\psi^{-1}(x')) 
		+ V_i(\psi^{-1}(x')) \right]  t' \di_{x'_i} 	\\
	&\quad + k (x'_1)^{k+1} V_1(\psi^{-1}(x')) (t')^2 \di_{t'}.
\end{align*}
Note that the affine hyperplane $\{z_0=0,\ z_1\ne 0\}$ 
corresponds to $\{x'_1=0\}$. 
Since $\psi^{-1}$ is a fractional linear map with a simple
pole along $x'_1=0$, the functions $V_i(\psi^{-1}(x'))$ are rational
in $x'$ with a pole of degree at most $k_0$ along $x'_1=0$ 
and no other singularities. It follows from the above expression 
that for $k\ge k_0$ the vector field $\wt V(x',t')$ is polynomial in $(x',t')$
and it vanishes on $\{x'_1=0\}\cup\{t'=0\}$.

In the general case when $V$ is given by \eqref{eq:V} and $m\in\N$
is arbitrary, the calculation is similar. Using the fibre
variables $t=(t_1,\ldots,t_m)$ on $E|U_0$ and
$t'=(t'_1,\ldots,t'_m)=x_1^k t$ on $E|U_1$, 
the differential $D\theta_{0,1}$ has the block form~\eqref{eq:Dtheta},
where $B$ is now an $m\times n$ matrix and the lower right 
entry is replaced by $x_1^{k}I_m$. The vector field 
$V'=D\theta_{1,0} \, \cdotp V$ equals
\begin{align*}
	V' &= \sum_{i=1}^n \sum_{j=1}^m t_j\, V_{i,j}(x) \di_{x_i} 
	  	 + k x_1^{k-1} \sum_{j,l=1}^m t_jt_l V_{1,j}(x) \di_{t'_l} \\
	   &= x_1^{-k} \sum_{i=1}^n \sum_{j=1}^m t'_j\, V_{i,j}(x) \di_{x_i} 
	   	 + k x_1^{-k-1} \sum_{j,l=1}^m t'_j t'_l V_{1,j}(x) \di_{t'_l}.
\end{align*}
The second step, expressing $V'$ in the affine coordinates $x'$
on $U_1$, is the same as before, and we leave 
the details to the reader. The new vector field 
$\wt V=D\tilde\psi\,\cdotp V'$ 
is polynomial in $(x',t')$, of second order in $t'$, and   
it vanishes on $\{x'_1=0\}\cup\{t'=0\}$.
Since this argument holds on every chart $U_\alpha$ for 
$\alpha=1,\ldots,n$, the lemma is proved.
\end{proof}

Since the extended vector field $V$ on $E$, given by Lemma \ref{lem:V}, 
vanishes on the zero section $E_0$ of $E$, 
there is a neighbourhood $\Omega \subset E$ of $E_0$ with convex fibres
such that the flow $\phi_\tau(e)$ of $V$, starting at time $\tau=0$ in 
any point $e\in \Omega$, exists for all $\tau \in [0,1]$. The map 
\[
	s=\pi\circ \phi_1:\Omega \to \CP^n
\]
is then a local holomorphic spray on $\CP^n$.
On the zero section $E_0\cong \CP^n$ we have 
a natural splitting $TE|E_0=E \oplus T\CP^n$.
Identifying a vector $e\in E_x=\pi^{-1}(x)$ with $e\in T_{0_x} E_x$, we let 
\[
	(V\!ds)_{x}(e)=(ds)_{0_x}(e)\in T_x \CP^n
\]
denote the vertical derivative of $s$ at $x\in \CP^n$ applied to the vector $e$.
We claim that for every $e=(x,t) \in \Omega$, with $x\in U_0$, we have 
\begin{equation}\label{eq:Vds}
	(V\!ds)_{x}(t_1,\ldots,t_m)= \sum_{j=1}^m t_j\, W_j(x).
\end{equation}
To see this, note that in the vector bundle chart on $E|U_0$ the 
vector field $V$ is of the form \eqref{eq:V},
that is, it is horizontal and its coefficients are linear in the fibre variable $t$.
It follows that 
\begin{equation}\label{eq:linear}
	\pi\circ \phi_\tau(x, \delta t) = \pi\circ \phi_{\delta \tau}(x,t)
\end{equation}
holds for every $(x,t)\in E|U_0$, $0\le \delta \le 1$, and all $\tau$ for which 
the flow exists. Taking $(x,t) \in \Omega$, this holds for all $\tau\in [0,1]$.
At $\tau=1$ we obtain 
\[
	s(x,\delta t) = \pi \circ \phi_1(x,\delta t) = 
	\pi \circ \phi_\delta(x,t),
	\quad 0\le \delta\le 1. 
\]
Differentiating with respect to $\delta$ at $\delta=0$ and noting that 
$\dfrac{d}{d\delta}\Big|_{\delta =0}\phi_\delta(x,t) = V(x,t)$
and $d\pi_{(x,t)} V(x,t) = \sum_{j=1}^m t_j\, W_j(x)$ (see \eqref{eq:tangent})
gives \eqref{eq:Vds}.

Set $E|Y=\pi^{-1}(Y)$. Condition \eqref{eq:tangent} implies that the spray
$s=\pi\circ \phi_1$ maps the domain $\Omega \cap E|Y$ to $Y$, 
so it is a local holomorphic spray on $Y$. Since the vector fields 
$W_1,\ldots,W_m$ generate the tangent space $T_x Y$ every point 
$x\in Y_0 = Y\cap U_0$, we see from \eqref{eq:Vds} that 
the restricted spray $s:\Omega \cap E|Y\to Y$ is dominating on $Y_0$.
On the other hand, since $V$ vanishes on $E|\Lambda_0$, 
$\phi_1$ is the identity on this set and the 
spray $s=\pi$ is trivial over $\Lambda_0$.

In order to find a local dominating spray on $Y$, we proceed as follows.
For $\alpha\in\{0,1,\ldots,n\}$ set $Y_\alpha=Y\cap U_\alpha$;
this is an algebraic submanifold of $U_\alpha\cong\C^n$. 
Choose $m\in\N$ big enough that the tangent bundle $TY_\alpha$
is pointwise generated by $m$ polynomial vector fields 
$W^\alpha_j$ of the form \eqref{eq:Wj} on $U_\alpha$ for every 
$\alpha\in\{0,1,\ldots,n\}$. In the affine coordinates
$x=(x_1,\ldots,x_n)=(z_0/z_\alpha,\ldots,z_n/z_\alpha)$ on $U_\alpha$ 
we have $W^\alpha_j(x) = \sum_{i=1}^n V^\alpha_{i,j}(x) \di_{x_i}$
where $V^\alpha_{i,j}$ are polynomials. Let 
\begin{equation}\label{eq:k0}
	k_0 := \max_{\alpha,i,j} \deg V^\alpha_{i,j}. 
\end{equation}
For every $k\ge k_0$ the above argument gives an algebraic vector field
$V^\alpha$ on the vector bundle $E^\alpha=m\U^k$ 
that vanishes on the zero section $E^\alpha_0$ and is 
of the form \eqref{eq:V} in the 
chart $E^\alpha|U_\alpha\cong U_\alpha\times \C^m$.
Explicitly, in the affine coordinates
$x=(z_0/z_\alpha,\ldots,z_n/z_\alpha)$ on $U_\alpha$
and fibre coordinates $t^\alpha=(t^\alpha_1,\ldots,t^\alpha_m)$ 
on $E^\alpha|U_\alpha$ we have 
\[ 
	V^\alpha(x,t^\alpha) = 
	\sum_{i=1}^n \sum_{j=1}^m t^\alpha_j \, V^\alpha_{i,j}(x) \di_{x_i}.
\] 
We can take $V^0$ to be the vector field $V$ in \eqref{eq:V}.

Let $E=E^0\oplus E^1\oplus \cdots\oplus E^n=(n+1)m\U^k$ 
and denote the vector bundle projection by $\pi:E \to \CP^n$.
The algebraic vector field $V^\alpha$ on $E^\alpha$ constructed above
can be extended to an algebraic vector field on $E$ by first extending
it trivially (horizontally) to each of the summands $E^\beta|U_\alpha$
of $E|U_\alpha$ for $\beta\ne \alpha$ (note that these are trivial bundles), 
and then observing that the resulting vector field on $E|U_\alpha$ 
extends to an algebraic vector field on $E$ taking into account the condition \eqref{eq:k0} (see Lemma \ref{lem:V}). With these extensions in place, 
we consider the vector field $V=\sum_{\alpha=0}^n V^\alpha$ on $E$.
The construction implies that
\[ 
	d\pi_e V(e)\in T_yY\quad 
	\text{for every $y\in Y$ and $e\in E_y=\pi^{-1}(y)$}.
\] 
Since each $V^\alpha$ vanishes on the zero section
of $E_0$ of $E$, so does $V$. 
Hence, there is a neighbourhood $\Omega\subset E$
of $E_0$ with convex fibres such that the flow $\phi_\tau(e)$ of $V$ 
exists for any initial point $e\in \Omega$ and every $\tau\in[0,1]$.
Consider the holomorphic spray 
\begin{equation}\label{eq:localspray}
	s=\pi \circ \phi_1:\Omega\to \CP^n. 
\end{equation}
We claim that $s:\Omega\cap \pi^{-1}(Y)\to Y$ is dominating. 
To see this, consider the vector field $V=\sum_{\alpha=0}^n V^\alpha$ 
on a chart $E|U_\beta$.
For simplicity of notation we assume that $\beta=0$;
the argument will be the same in every case. In the affine coordinates 
$x=(z_1/z_0,\ldots,z_n/z_0)$ on $U_0$ and fibre coordinates
$t=(t^0,t^1,\ldots,t^n)$ on $E|U_0$, where 
$t^\alpha = (t^\alpha_1,\ldots,t^\alpha_m)$ are 
fibre coordinates on the direct summand $E^\alpha|U_0$ 
of $E|U_0$, we see as above 
that 
\begin{equation}\label{eq:Vsum}
	V(x,t) = \sum_{\alpha=0}^n \sum_{i=1}^n 
	\sum_{j=1}^m t^\alpha_j \, \wt V^\alpha_{i,j}(x) \di_{x_i}
	+ \Upsilon(x,t)
	= \Theta(x,t)+\Upsilon(x,t)
\end{equation}
where each $\wt V^\alpha_{i,j}(x)$ is a polynomial, 
$\wt V^0_{i,j}(x)=V^0_{i,j}(x)$ for $i=1,\ldots,n$ and $j=1,\ldots,m$,  
and the vertical component $\Upsilon$ of $V$ vanishes to the second 
order along the zero section $E_0=\{t=0\}$, that is, $|\Upsilon(x,t)|=O(|t|^2)$. 
Since the vector field $\Theta(x,t)$ in \eqref{eq:Vsum}
is linear in the fibre variable $t$, its flow $\psi_\tau$ satisfies 
$\pi\circ \psi_\tau(x, \delta t) = \pi\circ \psi_{\delta \tau}(x,t)$
(cf.\ \eqref{eq:linear}). As before, it follows that the vertical
derivative of the spray 
$\tilde s=\pi\circ\psi_1:\Omega\to \CP^n$ over $U_0$ equals 
\[ 
	(V\! d\tilde s)_{x}(x,t)= 
	\sum_{\alpha=0}^n \sum_{j=1}^m t^\alpha_j\, \wt W^\alpha_j(x),
\] 
where each $\wt W^\alpha_j(x)$ is a polynomial vector field 
tangent to $Y_0$ and $\wt W^0_j=W^0_j$ for $j=1,\ldots,m$. 
(Compare with \eqref{eq:Vds}.) Since the vectors 
$W^0_j(x)$ for $j=1,\ldots,m$ 
span $T_xY$ for every $x\in Y_0$, the spray 
$\tilde s$ is dominating over $Y_0$. Since the second term
$\Upsilon$ in \eqref{eq:Vsum} is of size $O(|t|^2)$, 
a standard argument using Gr\"onwall's inequality
shows that the flow $\phi_\tau$ of $V$ satisfies 
\[
	\phi_\tau(x,t) = \psi_\tau(x,t)+O(|t|^2)\ \ \text{as $|t|\to 0$ and 
	$\tau\in [0,1]$}.
\]
It follows that the spray $s=\pi\circ \phi_1:\Omega\to\CP^n$
\eqref{eq:localspray} satisfies $(V\! ds)_{x}(x,t)=(V\! d\tilde s)_{x}(x,t)$ 
for $x\in U_0$. The same argument holds on every chart $E|U_\beta$.

This proves that the spray $s:\Omega\cap \pi^{-1}(Y)\to Y$ 
is dominating as claimed. 

%
%
We now replace the bundle $E\to \CP^n$ 
by its restriction $X:=E|Y \to Y$ and the spray $s$ 
\eqref{eq:localspray} by its restriction $X \cap \Omega\to Y$.
Since $X$ is a direct sum of copies 
of the negative line bundle $\U^k|Y$, it is a 1-convex manifold 
with the zero section $X_0\cong Y$ as the exceptional subvariety; 
that is, $X_0$ is the maximal compact complex subvariety 
of $X$ without point components;
see Grauert \cite[Satz 1, p.\ 341]{Grauert1962MA}.
Such $X$ admits a plurisubharmonic exhaustion function
$\rho:X\to[0,\infty)$ which vanishes on $X_0$ and is positive 
strongly plurisubharmonic on $X\setminus X_0$, and the Remmert
reduction of $X$ is a Stein space. 

So far, we have not used the hypothesis that $Y$ is an Oka manifold.
Under this additional assumption, we shall prove that there
exists a holomorphic spray $\tilde s: X \to Y$ which agrees
with $s$ to the second order along the zero section $X_0\cong Y$
of $X$. Clearly, $\tilde s$ is then a dominating spray on $Y$, 
thereby proving Theorem \ref{th:elliptic}.

Since $X$ is a 1-convex manifold and $Y$ is an Oka manifold,
the existence of such $\tilde s$ follows from 
the main result of Prezelj \cite{Prezelj2010} if taken at face value; 
see also \cite{Prezelj2016} which clarifies 
a part of the construction in \cite{Prezelj2010}. 
However, it was brought to our attention that one of the 
technical tools, \cite[Theorem 2.4]{Prezelj2010},  
claims more than had been proved in the literature,
and possibly there exists a counterexample to this statement.
(This is quoted from Grauert's paper \cite{Grauert1962MA} 
but certain conditions in his result were not taken into account.) 
For this reason, we proceed in a different way. 
The main point is to prove the following lemma.

%
%
\begin{lemma}\label{lem:localspray}
(Assumptions as above.) Let $r\ge 2$ be an integer.
Assume that $U\subset X$ is an open neighbourhood of 
the zero section $X_0\cong Y$ of $\pi:X\to Y$ 
and $s:U\to Y$ is a local dominating holomorphic spray.
There are an open neighbourhood $U'\subset U$ 
of $X_0$, a ball $0\in B\subset \C^N$ for some $N\in\N$, 
and a holomorphic spray of maps
$S:U'\times B \to Y$ with the core $S(\cdotp,0)=s|U'$
such that $S$ agrees with $s$ to order $r$ along $X_0$ for every 
$\zeta\in B$, and $S$ is dominating over $U'\setminus X_0$
with respect to the variable $\zeta \in B$.
\end{lemma}

More precisely, the last claim is that for every point 
$x\in U'\setminus X_0$ the differential of the map 
$B\ni \zeta \mapsto S(x,\zeta)\in Y$ at $\zeta =0$
maps $T_0 B =\C^N$ onto $T_{s(x)}Y$.
(Note that $S(x,0)=s(x)$.) 
If Lemma \ref{lem:localspray} holds then  
the existence of a holomorphic map $\tilde s:X\to Y$ which agrees
with $s$ on $X_0$ to order $r$ follows from 
\cite[Theorem 3.3]{Stopar2013} by Stopar. 
Indeed, Lemma \ref{lem:localspray} shows that 
every map $s:U\to Y$ as in the lemma
satisfies Condition $\Ecal$ in \cite[p.\ 4]{Stopar2013}. 
Assuming as we may that the domain $U\Subset X$ is 
relatively compact with smooth strongly pseudoconvex boundary, 
\cite[Theorem 2.5]{Stopar2013} gives  
a holomorphic spray $\wt S:U\times B\to Y$ with the
core $\wt S(\cdotp,0)=s$ which is dominating over $U\setminus X_0$ 
with respect to the $\zeta$ variable and such that 
$\wt S(\cdotp,\zeta)$ agrees with $s$ to order $r$ 
along $X_0$ for every $\zeta\in B$. 
(Assuming that $N\ge \dim X+\dim Y$, $\wt S$ may be chosen to approximate the spray $S$ in the hypotheses of the lemma over a neigbourhood of $X_0$.) 
The existence of a local dominating spray $\wt S$ with these properties was 
the main technical issue addressed in \cite{Prezelj2010,Prezelj2016}. 
With this result in hand, one can apply the usual inductive procedure 
in Oka theory, adapted to 1-convex manifolds, 
to construct a global holomorphic map $\tilde s:X \to Y$ which agrees 
with $s$ to order $r$ along $X_0$. See \cite[Theorem 3.3]{Stopar2013}
and its proof for the details. (The cited result applies 
to sections of any holomorphic fibre bundle with an Oka fibre over $X$,
but we apply it to maps $X\to Y$ identified with sections of the 
trivial bundle $X\times Y\to X$.)

\begin{proof}[Proof of Lemma \ref{lem:localspray}]
The idea is to precompose the map $s:U\to Y$ by a suitably 
chosen spray of holomorphic fibre preserving 
self-maps of the vector bundle $\pi:X\to Y$. 
Recall that this bundle is trivial over 
each domain $Y\cap U_j$ for $j=0,1,\ldots,n$, where 
$U_j\subset\CP^n$ are the standard affine charts.
We explain the construction over $Y_0=Y\cap U_0$; it will
apply by symmetry to all other charts.
Let $x=(z_1/z_0,\ldots,z_n/z_0)$ be 
the affine coordinates on $U_0$ and $t=(t_1,\ldots,t_{l})$
the fibre coordinates on the vector bundle chart 
$X|Y_0\cong Y_0\times \C^{l}$ with $l=m(n+1)$. 
Let $r \ge 2$ be as in the lemma. 
Choose $v\in \C^{l}$ 
and consider the holomorphic map 
$\phi:Y_0 \times \C^l \times \C \to Y_0 \times \C^l$ given by
\begin{equation}\label{eq:phi}
	\phi(x,t,\zeta) = (x,t+ \zeta t_1^r v).
\end{equation}
Note that $\phi(x,t,0)=(x,t)$. We claim that for every $\zeta\in\C$, 
$\phi(\cdotp,\cdotp,\zeta)$ extends to a holomorphic 
fibre preserving self-map of $X$ which agrees 
with the identity to order $r$ along $X_0$.
Indeed, consider its expression in the  
vector bundle chart $X|Y_\alpha$ for some $\alpha\in \{1,\ldots,n\}$.
The fibre coordinates $t'$ on this chart are related to $t$ 
by $t'=(z_\alpha/z_0)^k t$ (see \eqref{eq:transition}), 
and a calculation gives
\[
	\phi(x,t',\zeta)=\left(x,t'+ \zeta (t'_1)^r (z_0/z_\alpha)^{k(r-1)} v\right)
	= \left(x,t'+ \zeta (t'_1)^r (1/x_\alpha)^{k(r-1)} v\right). 
\]
This shows that for every $\zeta\in\C$ 
the map $\phi(\cdotp,\cdotp,\zeta)$ agrees with the identity on $X$ 
over the affine hyperplane $\{z_0=0,\ z_\alpha\ne 0\}$ intersected with $Y$, 
and it agrees with the identity to order $r\ge 2$ along 
$\{t=0\}=X_0$. Since this holds for every $\alpha=1,\ldots,n$, 
the claim follows. Identifying the tangent bundle $T(X|Y_\alpha)$ with 
$TY_\alpha \times T\C^l$, we have that 
\[
	\frac{\di}{\di\zeta}\Big|_{\zeta=0} \phi(x,t,\zeta) = (0,t_1^r v)
\]
and hence
\begin{equation}\label{eq:sphi}
	\frac{\di}{\di\zeta}\Big|_{\zeta=0} s\circ \phi(x,t,\zeta)=
	ds_{(x,t)} (0,t_1^r v) = t_1^r ds_{(x,t)}(0,v).
\end{equation}
Since the spray $s$ is dominating by the assumption, 
we can choose $v\in \C^l$ such that 
$ds_{(x,0)}(0,v)$ equals any given tangent vector $w\in T_x Y$,
and hence the vector \eqref{eq:sphi} equals $t_1^r w$.
Let $\phi_1,\ldots,\phi_d$ for $d=\dim Y$ be sprays of the form
\eqref{eq:phi} for vectors $v_1,\ldots,v_d\in \C^l$ chosen such that 
the vectors $ds_{(x,0)}(0,v_i)$ for $i=1,\ldots,d$ form a basis
of $T_x Y$. Then, there are a neighbourhood $B'\subset\C^d$
of the origin and a neighbourhood $U'\subset U$ of $X_0$
such that the map
$
	s \circ \phi_1\circ\cdots\circ\phi_d : U' \times B' \to Y
$
given by 
\[
	(x,t,\zeta_1,\ldots,\zeta_d) \longmapsto 
	s \circ \phi_1(x,t,\zeta_1)\circ\cdots \circ \phi_d(x,t,\zeta_d)
\]
is a local spray with the core $s$ at $\zeta=0$ 
which is dominating with respect to  
$\zeta=(\zeta_1,\ldots,\zeta_d)$ at $\zeta=0$ 
in a neighbourhood of $(x,0)\in X_0$ in $X$, except on the
hyperplane $t_1=0$, and for every fixed $\zeta\in B'$
it agrees with $s$ to order $r$ along 
the zero section $X_0\cong Y$. Repeating this construction with $t_1^r$
replaced by $t_j^r$ for $j=1,\ldots,l$ and at other points $x\in Y$
(also in other charts on $Y$)
gives finitely many holomorphic sprays $\phi_1,\ldots, \phi_N$ on $X$ 
of the form \eqref{eq:phi} and neighbourhoods $U'\subset U$ of $X_0$
and $0\in B\subset \C^N$ such that the map 
\[
	S=s \circ \phi_1\circ\cdots\circ\phi_N : U' \times B \to Y
\]
satisfies the conclusion of the lemma.
\end{proof}

\begin{remark}\label{rem:maintheorem}
In the proof of Theorem \ref{th:elliptic}, we 
begin with suitably chosen algebraic (polynomial) vector fields
on affine vector bundle charts, which extend 
to algebraic vector fields on a sufficiently negative vector bundle 
on the given manifold $Y$. 
This is only possible on projective manifolds
since a compact complex manifold with a negative 
(or a positive) line bundle is necessarily projective
according to Kodaira \cite{Kodaira1954}. 
The subsequent techniques using flows of vector fields, 
and especially the last step of the proof to construct
a global dominating spray, are transcendental.
Hence, this method does not give algebraic ellipticity.
This is not surprising, for we have already mentioned that 
there are examples of projective Oka manifolds 
that fail to be algebraically elliptic, for example, abelian varieties.
\end{remark}

%
%
\section{Further results and remarks on ellipticity}
\label{sec:remarks}

In this section we collect some further observations
concerning the relationship between the Oka property and 
ellipticity of a complex manifold.

\begin{remark}\label{rem:elliptic}
If $L\to Y$ is a negative holomorphic line bundle on a compact 
(hence projective) manifold $Y$, 
then for sufficiently large $k>0$ the vector bundle 
$\Hom(L^k,TY)\cong L^{-k}\otimes TY$ on $Y$ is generated by finitely 
many global holomorphic sections $h_1,\ldots, h_N$
(a theorem of Hartshorne; 
see Lazarsfeld \cite[Theorem 6.1.10]{Lazarsfeld-2-2004}). 
Let $E=N L^k$ denote the direct sum of $N$ copies of $L^k$.
Considering $h_i$ as a homomorphism 
$h_i:L^k\to TY$, it follows that the holomorphic vector bundle map
$h = \oplus_{i=1}^N h_i : E\to TY$ is an epimorphism.
Gromov proposed \cite[3.2.A', Step 2, p.\ 879]{Gromov1989}
that such $h$ is the vertical derivative of a local dominating holomorphic 
spray $s:U\to Y$ from an open neighbourhood $U\subset E$
of its zero section $E_0\cong Y$. This would give a shorter
proof of Theorem \ref{th:elliptic}.
Although we do not know how to justify Gromov's claim,
our proof of Theorem \ref{th:elliptic} follows this idea in spirit 
if not to the letter. This raises the following question.
\end{remark}

\begin{problem} \label{prob:localspray}
Which holomorphic vector bundles $E\to Y$ of 
$\mathrm{rank}\,E\ge \dim Y$ 
admit a local dominating spray $s:U\to Y$ from a neighbourhood
$U\subset E$ of the zero section of $E$?
\end{problem}

The following observation generalises 
\cite[Proposition 6.2]{Forstneric2019MMJ}. 
Recall that every complex homogeneous manifold is 
elliptic \cite[Proposition 5.6.1]{Forstneric2017E}, and hence 
an Oka manifold.

\begin{proposition}\label{prop:homogeneous}
Assume that a compact complex manifold $Y$ admits a 
local dominating holomorphic spray $(E,\pi,s)$. If the 
vector bundle $\pi:E\to Y$
is generated by global holomorphic sections, 
then $Y$ is a complex homogeneous manifold. 
\end{proposition}

The condition on $E$ to be globally generated by holomorphic 
sections holds for a trivial bundle 
and for any sufficiently Griffiths positive bundle, 
but fails for negative bundles.

\begin{proof}
Let $s:U\to Y$ be a local dominating spray defined on a
neighbourhood $U\subset E$ of the zero section $E_0$. 
The vertical derivative $V\!ds|E_0: VT(E)|E_0= E \to TY$ 
is a vector bundle epimorphism. 
Given a holomorphic section $\xi : Y\to E$, the map  
\[
	Y\ni y \longmapsto V_\xi(y) := V\!ds(y)(\xi(y)) \in T_yY
\]
is a holomorphic vector field on $Y$. 
(We are using the natural identification of the vertical tangent
bundle $VT(E)|E_0$ on the zero section $E_0$ with the bundle $E$ itself.)
Applying this argument to sections $\xi_1,\ldots,\xi_m:Y\to E$
generating $E$ gives holomorphic vector fields 
$V_1,\ldots,V_m$ on $Y$ spanning the tangent 
bundle $TY$ since $V\!ds$ is surjective. 
Thus, $Y$ is holomorphically flexible.
Since $Y$ is compact, these vector fields are complete, so 
their flows are complex $1$-parameter subgroups of the holomorphic automorphism group $\Aut\,Y$, which is a finite-dimensional complex 
Lie group \cite{BochnerMontgomery1947-2}.
The spanning property implies that $\Aut\,Y$ acts transitively on $Y$,
so $Y$ is homogeneous.
\end{proof}

There are projective Oka manifolds that are not homogeneous, 
for instance, blowups of certain projective manifolds such as projective spaces, 
Grassmannians, etc.; see \cite[Propositions 6.4.5 and 6.4.6]{Forstneric2017E}, 
the papers \cite{KalimanKutzschebauchTruong2018,LarussonTruong2017},
and the survey \cite[Subsect.\ 6.3]{Forstneric2023Indag}.
Many of these manifolds are algebraically elliptic.
Another class of non-homogeneous projective surfaces that 
are algebraically elliptic are the Hirzebruch surfaces $H_l$ for $l=1,2,\ldots$; 
see \cite[p.\ 191]{BarthHulek2004} and 
\cite[Proposition 6.4.5]{Forstneric2017E}. 
In view of Proposition \ref{prop:homogeneous}, 
such manifolds do not admit a local dominating spray 
from any globally generated holomorphic vector bundle.

\begin{remark}
Let $\mathscr S$ be the largest class of complex manifolds for which 
the Oka property implies ellipticity, that is, the class of manifolds that are 
either elliptic or not Oka.  As remarked above, it is long known that every 
Stein manifold belongs to $\mathscr S$.  By Theorem \ref{th:elliptic}, so 
does every projective manifold.  We know of two ways to produce new 
members of $\mathscr S$ from old.  If $Y\to X$ is a covering map and $X$ is elliptic, so is $Y$. Also, $X$ is Oka if and only if $Y$ is.  Hence, a covering space of a manifold in $\mathscr S$ is in $\mathscr S$.  Also, it is easily seen that a product of manifolds in $\mathscr S$ is in $\mathscr S$.
\end{remark}

%
%
\section*{Acknowledgements}
Forstneri\v c is supported by the European Union 
(ERC Advanced grant HPDR, 101053085) 
and grants P1-0291 and N1-0237 from ARIS, Republic of Slovenia. 
A part of the work on this paper was done during his visit to the University
of Adelaide in February 2025, and he wishes to thank the institution for 
hospitality. The authors thank Yuta Kusakabe for having brought
to our attention the problem with the paper \cite{Prezelj2010}  
mentioned before Lemma \ref{lem:localspray}.


\end{document}